\documentclass{elsarticle}

\usepackage{color}
\usepackage{amsmath}
\usepackage{amsfonts}
\usepackage{amssymb}
\usepackage{amsthm}

\newtheorem{lemma}{Lemma}

\newcommand{\ee}{\mathrm{e}}
\newcommand{\ad}{\mathrm{ad}}
\newcommand{\nL}{\mathcal{L}}
\newcommand{\QQ}{\mathbb{Q}}
\newcommand{\nA}{\mathcal{A}}

\begin{document}
\begin{frontmatter}
\title{A relatively short self-contained proof of the Baker-Campbell-Hausdorff theorem}

\author{Harald Hofst\"atter}%\corref{ourcorrespondingauthor}}
%\cortext[ourcorrespondingauthor]{Corresponding author}
\address{Universit{\"a}t Wien, Institut f{\"u}r Mathematik, Oskar-Morgenstern-Platz 1, A-1090 Wien, Austria}
\ead{hofi@harald-hofstaetter.at}
\ead[url]{www.harald-hofstaetter.at}


\begin{abstract}We give a new purely algebraic proof of the Baker-Campbell-Hausdorff theorem, which states that the homogeneous components of the formal
expansion of $\log(\ee^A\ee^B)$ are Lie polynomials.
Our proof is based on a recurrence formula for these components and a  lemma that states that if under certain conditions a commutator of a non-commuting variable and a given polynomial is a Lie polynomial, then the given polynomial itself is a Lie polynomial.
\end{abstract}

\end{frontmatter}

\sloppy
\section{Introduction}\label{Sec:intro}
Let $\QQ\langle\langle\nA\rangle\rangle$ denote the ring of formal power series with rational coefficients in  non-commuting variables from a set $\nA$.\footnote{In the following we assume that the basic field is $\QQ$, because this is the natural setting for computations around the Baker-Campbell-Hausdorff theorem. We note that 
everything  remains valid if 
$\QQ$ is consistently replaced by an extension field $K\supset\QQ$ like $\mathbb{R}$
or $\mathbb{C}$.
}
We  assume that the set $\nA$ is finite and contains at least two elements.
 In this ring exponential and logarithm functions 
%$\exp,\log:\QQ\langle\langle\nA\rangle\rangle\to\QQ\langle\langle\nA\rangle\rangle$
are defined by their power series expansions
\begin{equation*}
\exp(X)=\ee^X=\sum_{n=0}^\infty\frac{1}{n!}X^n,\quad
\log(1+X)=\sum_{n=1}^\infty\frac{(-1)^{n+1}}{n}X^n
\end{equation*}
for elements $X$ of $\QQ\langle\langle\nA\rangle\rangle$ without constant term.
For two non-commuting variables $A$, $B$ we consider the formal
expansion of 
\begin{equation}\label{eq:BCH}
 C=\log(\ee^A\ee^B)=C_1+C_2+\dots
\end{equation}
in $\QQ\langle\langle A,B\rangle\rangle$,
where for each $n\geq 1$ all terms of degree $n$ are collected 
in the homogeneous polynomial $C_n$ of degree $n$.
The Baker-Campbell-Hausdorff (BCH) theorem states that each of these 
polynomials $C_n$ can be written as a Lie polynomial, i.e.,
 a linear combination of $A$ and $B$ and (possibly nested) commutator terms in $A$ and $B$, see \cite[Section~7.6]{stillwell2010naive}.

The record for the shortest self-contained proof of the BCH theorem is presumably
held by M.~Eichler with his ingenious two-page proof \cite{eichler1968},
see also \cite[Section~7.7]{stillwell2010naive}.
This purely algebraic proof  requires almost no theoretical preparation, 
but, on the other hand, it does not
provide deeper insights into {\em why} the BCH theorem is true.
%It is this last aspect that makes us think that our own approach to a proof is %of some value. 
In this last aspect, our own approach to a proof seems to be more illuminating.
 
\section{Towards a proof of the BCH theorem} 
We introduce some notations.
Let $\QQ\langle\nA\rangle\subset\QQ\langle\langle\nA\rangle\rangle $ denote  the subring of all 
formal power series
with only finitely many %terms with 
non-zero coefficients,
i.e., the ring of polynomials in the
non-commuting variables $\nA$.
$\QQ\langle\nA\rangle$ considered as a vector space over
the field $\QQ$ with Lie bracket defined by the commutator
%\begin{equation*} %\label{eq:Lie_bracket}
$[X,Y]=XY-YX$
%\end{equation*}
is a Lie algebra, i.e., the Lie bracket is bilinear and satisfies
\begin{align*}
&[X,Y]=-[Y,X]\qquad\qquad\qquad\qquad\qquad\qquad\ \mbox{(anti-symmetry)}\\
&[X,[Y,Z]]+[Z,[X,Y]]+[Y,[Z,X]]=0
\qquad\, \mbox{(Jacobi identity)}
\end{align*}
for all $X,Y,Z\in\QQ\langle\nA\rangle$, see \cite[Section~4.4]{stillwell2010naive}. We define  
$\nL_\QQ(\nA)\subset\QQ\langle\nA\rangle$ as the smallest subspace of 
$\QQ\langle\nA\rangle$  containing $\nA$ and being closed under Lie brackets.  
$\nL_\QQ(\nA)$ is called the free Lie algebra generated by $\nA$.
It is clear that the elements of $\nL_\QQ(\nA)$ are 
precisely the linear combinations of variables from $\nA$ and
simple and nested Lie brackets in these variables, i.e., the Lie polynomials
in the variables $\nA$.

We begin our considerations aiming at a proof of the BCH theorem with
the well-known formal identity
\begin{equation}\label{eq:basic_identity}
\ee^X Y \ee^{-X} = \ee^{\ad_X}(Y)
:=\sum_{n=0}^\infty\frac{1}{n!}\ad_X^n(Y).%, \quad X,Y\in\QQ\langle\langle \nA\rangle\rangle.
\end{equation}
Here $\ad_X$ is the linear operator defined by $Y\mapsto [X,Y]=XY-YX$ such that
$\ad_X^n(Y)=[X,[X,[\dots [X,Y]\dots]]$ with $X$ occurring $n$ times.
The proof of (\ref{eq:basic_identity}) is  easy:
We define linear operators $L_X$, $R_X$ 
by $L_X(Z)=XZ$, $R_X(Z)=ZX$ such that $\ad_X=L_X-R_X$.
These operators commute, because
$L_XR_X(Z)=XZX=R_XL_X(Z)$.
It follows
\begin{equation*}
\ee^{\ad_X}(Y)=\ee^{L_X-R_X}(Y)=\ee^{L_X}\ee^{-R_X}(Y)
=\ee^{L_X}\ee^{R_{-X}}(Y)
%=\ee^{L_X}(Y\ee^{-X})
=\ee^{X}Y\ee^{-X}.
\end{equation*}

Using $\ee^C=\ee^A\ee^B$ with $C$ defined by (\ref{eq:BCH}) and applying
(\ref{eq:basic_identity}) twice we obtain
\begin{equation*}
\ee^{\ad_C}(B)\ee^C=\ee^CB=\ee^A\ee^BB=\ee^AB\ee^B=\ee^{\ad_A}(B)\ee^C,
\end{equation*}
and thus, after cancelling the factor $\ee^{C}$,
\begin{equation*}
\sum_{k=0}^\infty\frac{1}{k!}\ad_C^k(B)
=\sum_{k=0}^\infty\frac{1}{k!}\ad_{C_1+C_2+\dots}^k(B)
=\sum_{k=0}^\infty\frac{1}{k!}\ad_A^k(B).
\end{equation*}
In this equation consider the terms of degree $n+1$, 
which leads to
  \begin{equation*}
    \sum_{m=1}^n\frac{1}{m!}\sum_{k_1+\dots+k_m=n\atop k_j\geq 1}
    \ad_{C_{k_1}}\!\circ\ad_{C_{k_2}}\!\circ\dots\circ
    \ad_{C_{k_m}}(B)
    =\frac{1}{n!}\ad_A^n(B),
  \end{equation*}
or, after splitting off the term for $m=1$ from the sum  and reordering,
  \begin{equation}\label{eq:BC_recursion}
    [B,C_n]=\sum_{m=2}^n\frac{1}{m!}\sum_{k_1+\dots+k_m=n\atop k_j\geq 1}
    \ad_{C_{k_1}}\!\circ\ad_{C_{k_2}}\!\circ\dots\circ
    \ad_{C_{k_m}}(B)
    -\frac{1}{n!}\ad_A^n(B).      
  \end{equation}
Here we make two observations:
\begin{enumerate}
\item Provided  equation (\ref{eq:BC_recursion}) can be uniquely solved for $C_n$, it can be used as a recurrence formula for the
computation of the $C_n$.
\item If we assume that all $C_k$ for $k<n$ are Lie 
polynomials in $A,\,B$, then it is obvious that
the right-hand side of (\ref{eq:BC_recursion}) is also a Lie polynomial. If we can
show that $C_n$ is a Lie polynomial provided $[B,C_n]$ is a Lie polynomial, 
then, using induction on $n$, this leads to a proof of the BCH theorem.
\end{enumerate}

With regard to  the first observation, 
the following lemma gives
a criterion for 
the uniqueness of a solution of (\ref{eq:BC_recursion}).
(Because $C_n$ defined by (\ref{eq:BCH}) is a solution,
there is no question about existence.)

\begin{lemma}\label{lemma:ad_injectivity}
If  $\ad_a(P)=[a,P]=0$ for $a\in\nA$ and $P\in \QQ\langle\nA\rangle$    not containing any term $\alpha_ka^k,\,\alpha_k\neq 0,\, k\geq 0$, then
$P=0$.
\end{lemma}
\begin{proof}
Towards a contradiction, we assume that there exists a word $w=w_1\cdots w_n$, $w_j\in\nA$ with non-zero coefficient $(P,w)$ in the expansion $P=\sum_w(P,w)w$.
$w$ must contain at least one letter $w_j\neq a$ and thus
has the form
$w=vxa^k$ with $k\geq 0$ for some word $v$ and $x\in\nA$, $x\neq a$.
Let $\widetilde{w}$ be a word with $(P,\widetilde{w})\neq 0$
of this form with $k$ maximal.
Then the coefficient $(\ad_a(P),\widetilde{w}a)$ of $\widetilde{w}a$ in
$\ad_a(P)=\sum_w(P,w)(aw-wa)$ is 
$(\ad_a(P),\widetilde{w}a)=-(P,\widetilde{w})$ because $aw\neq \widetilde{w}a$ for all words $w$ by the maximality of $\widetilde{w}$.
Because $(P,\widetilde{w})\neq 0$ this is a contradiction to
$\ad_a(P)=0$.
\end{proof}

We have thus established the uniqueness of the
solution of (\ref{eq:BC_recursion})
up to a term $\beta_nB^n$. By the following lemma no such term can 
actually occur if $n\geq 2$. Therefore, for $n\geq 2$ the solution of  
(\ref{eq:BC_recursion}) is indeed unique. Furthermore, the
lemma gives us the initial value $C_1=A+B$ for the recursion.

\begin{lemma}\label{lemma:noBBBinC}
In the expansion of 
$C=\log(\ee^A\ee^B)=A+B+[\mbox{terms of degrees $\geq 2$}]$ in $\QQ\langle\langle A,B\rangle\rangle$ there are no terms $\alpha_kA^k$, $\alpha_k\neq 0$, $k\geq 2$ or $\beta_kB^k$, $\beta_k\neq 0$, $k\geq 2$.
\end{lemma}
\begin{proof}
Write $\ee^A\ee^B=e^B+h(A,B)$ where each term in the expansion of
$h(A,B)=(e^A-1)e^B$ contains at least one $A$. Then
\begin{align*}
C&=\log\ee^A\ee^B =\sum_{n=1}^\infty\frac{(-1)^{n+1}}{n}((e^B-1)+h(A,B))^n
\\
&=\!\sum_{n=1}^\infty\!\frac{(-1)^{n+1}}{n}(e^B-1)^n+H(A,B)
=\log(\ee^B)+H(A,B) = B+H(A,B),
\end{align*}
where each term in the expansion of $H(A,B)$ contains at least one $A$.
A similar argument leads to $C=\log\ee^A\ee^B =A+\widetilde{H}(A,B)$,
where each term in the expansion of $\widetilde{H}(A,B)$ contains at least one $B$.
\end{proof}

With regard to the second of the above observations, it follows 
from 
Lemma~\ref{lemma:noBBBinC} and
Lemma~\ref{lemma:P_Lie_from_adP_Lie} below that for $n\geq 2$,
$C_n$ indeed is a Lie polynomial provided $[B,C_n]$ is a Lie polynomial.
As already mentioned, this together with $C_1=A+B$ completes the proof of the BCH theorem.

\section{A lemma crucial for our proof of the BCH theorem}
\label{Sec:techn}

\begin{lemma}\label{lemma:P_Lie_from_adP_Lie}
Let $a\in\nA$ and $P\in\QQ\langle\nA\rangle$ 
not containing any term $\alpha_ka^k$, $\alpha_k\neq 0$, $k\geq 0$.
If $\ad_{a}(P)=[a,P]$ 
is a Lie polynomial, 
then $P$ is a Lie polynomial. 
\end{lemma}

It is possible to  give a short
(but not  self-contained)
proof of  this lemma
by taking  strong results from the theory of free Lie algebras
like    \cite[Theorem~1.4]{reutenauer1993free} for granted.
The following  self-contained proof  was essentially obtained by
distilling  
from  \cite[Chapter~1]{reutenauer1993free}
just enough material to derive Lemma~\ref{lemma:P_Lie_from_adP_Lie}.

Our proof is based on properties of the right normed bracketing
% Remarkably,  a certain degree of sophistication seems to be unavoidable. 
which
%We define the right normed bracketing
 for words $w=w_1\cdots w_n$,
$w_j\in\nA$ is defined by $r(w)=[w_1,[w_2,[\dots[w_{n-1}, w_n]\dots]]]$.
Because the set of words is a basis of the vector space
$\QQ\langle\nA\rangle$, $r$ can  uniquely be extended
to a linear map $r:\QQ\langle\nA\rangle\to\QQ\langle\nA\rangle$
satisfying the basic identities 
\begin{equation*}
r(x)=x, \quad
r(xP)= \ad_x(r(P)),\qquad x\in\nA, \ P\in\QQ\langle\nA\rangle.
\end{equation*}
By recursive applications of the Jacobi identity written as
\begin{equation*}
[[X,Y],Z]=[X,[Y,Z]]-[Y,[X,Z]]
\end{equation*}
to a Lie polynomial $P\in\nL_\QQ(\nA)$, we eventually obtain
a linear combination of right normed elements
$[a_1,[a_2,[\dots[a_{n-1}, a_n]\dots]]]$, $n\geq 1$, $a_j\in\nA$.
This shows that for each Lie polynomial $P\in\nL_\QQ(\nA)$ there exists
a  polynomial $\widetilde{P}\in\QQ\langle\nA\rangle$ such that
$P=r(\widetilde{P})$. 

After the following sequence of Lemmas~\ref{lemma:baker_identity}--\ref{lemma:rP_nP} based on  \cite[Section~1.6.6]{reutenauer1993free}  
%and culminating in the simple but powerful Lemma~\ref{lemma:rP_nP} 
and  an auxiliary result from Lemma~\ref{lemma:rPa_ad_aP},
the proof of Lemma~\ref{lemma:P_Lie_from_adP_Lie}
will finally succeed in a relatively simple way.

\begin{lemma}\label{lemma:baker_identity}
For polynomials $P,\,Q\in\QQ\langle\nA\rangle$ we have
\begin{equation}\label{eq:baker_identity}
r(r(P)Q)=[r(P),r(Q)].
\end{equation}
\end{lemma}
\begin{proof}
Because both sides of the identity are linear in $P$, we
only have to consider the special case where $P$ is a word.
We use induction on the length $n$ of $P$. 
For $n=1$ we have $P\in\nA$ and (\ref{eq:baker_identity}) follows
immediately. 
For $n>1$ write $P=x\widetilde{P}$ with $x\in\nA$ and $\widetilde{P}$ a word of length $n-1$. Then,
\begin{align*}
&r(r(P)Q)=r(r(x\widetilde{P})Q)
=r([x,r(\widetilde{P})]Q)
=r(xr(\widetilde{P})Q)-r(r(\widetilde{P})xQ)\\
&\qquad=[x,r(r(\widetilde{P})Q)]-[r(\widetilde{P}),r(xQ)]\qquad\quad\mbox{(2nd term by induction hypothesis)}\\
&\qquad=[x,[r(\widetilde{P}),r(Q)]]-[r(\widetilde{P}),[x,r(Q)]]\quad\mbox{(1st term by induction hypothesis)}\\
&\qquad=[[x,r(\widetilde{P})],r(Q)]\qquad\qquad\qquad\qquad\quad\mbox{(by Jacobi identity)}\\
&\qquad=[r(P),r(Q)].\qedhere
\end{align*}
\end{proof}

\begin{lemma}\label{lemma:r_derivation}
For Lie polynomials $P_1,P_2\in\nL_\QQ(\nA)$ we have
\begin{equation*}
r([P_1,P_2])=[P_1,r(P_2)]+[r(P_1),P_2].
\end{equation*}
\end{lemma}
\begin{proof}
There exist  polynomials 
$\widetilde{P}_1,\,\widetilde{P}_1$ such that
$P_1=r(\widetilde{P}_1),\, P_2=r(\widetilde{P}_2)$.
From Lemma~\ref{lemma:baker_identity} it follows
\begin{align*}
&r([P_1,P_2])=r([r(\widetilde{P}_1),r(\widetilde{P}_2)]) 
=r(r(\widetilde{P}_1)r(\widetilde{P}_2))-r(r(\widetilde{P}_2)r(\widetilde{P}_1))\\
&\qquad=[r(\widetilde{P}_1),r^2(\widetilde{P}_2)]-[r(\widetilde{P}_2),r^2(\widetilde{P}_1)]
=[r(\widetilde{P}_1),r^2(\widetilde{P}_2)]+[r^2(\widetilde{P}_1),r(P_2)]\\
&\qquad=[P_1,r(P_2)]+[r(P_1),P_2].\qedhere
\end{align*}
\end{proof}

\begin{lemma}\label{lemma:rP_nP}
If $P\in\nL_\QQ(\nA)$ is a homogeneous Lie polynomial of degree $n$,
then
\begin{equation}\label{eq:rP_nP}
r(P) = nP.
\end{equation}
\end{lemma}
\begin{proof}
We use induction on $n$.
For $n=1$, $P=\sum_{a\in\nA}\alpha_a a$  and
(\ref{eq:rP_nP}) follows immediately.
If $n>1$, then $P$ is a linear combination of elements of the form  $[P_1,P_2]$ 
where $P_1,\,P_2$ are 
homogeneous Lie polynomials of degrees $n_1,\,n_2\geq 1$ with
$n_1+n_2=n$. For such elements we have
\begin{equation*}
r([P_1,P_2]) = [r(P_1),P_2]+[P_1,r(P_2)]=n_1[P_1,P_2]+n_2[P_1,P_2]=n[P_1,P_2],
\end{equation*}
where we used Lemma~\ref{lemma:r_derivation} and the induction hypothesis $r(P_1)=n_1P_1$,
$r(P_2)=n_2P_2$ for $n_1,n_2<n$.
Using  linearity of $r$ we obtain (\ref{eq:rP_nP}).
\end{proof}

\begin{lemma}\label{lemma:rPa_ad_aP}
If $a\in\nA$ and $P\in\nL_\QQ(\nA)$ is a Lie polynomial,
then
\begin{equation*}
r(Pa) = -\ad_a(P).
\end{equation*}
\end{lemma}
\begin{proof}By linearity, it suffices to consider the case where $P$
is homogeneous of degree $n\geq 1$. 
Using Lemma~\ref{lemma:rP_nP} we obtain
\begin{align*}
r(Pa)&=r(aP)-r(\ad_a(P))=\ad_a(r(P))-r(\ad_a(P))\\
&=n\,\ad_a(P)-(n+1)\ad_a(P)=-\ad_a(P).\qedhere
\end{align*}
\end{proof}

\begin{proof}[Proof of Lemma~\ref{lemma:P_Lie_from_adP_Lie}]
Let $a\in\nA$ and $P\in \QQ\langle\nA\rangle$  not containing any term $\alpha_ka^k,\,\alpha_k\neq 0,\, k\geq 0$ such that
$\ad_a(P)$ is a Lie polynomial. We want to show that $P$ is 
a Lie polynomial. 
Splitting $P$ into homogeneous components 
and $\ad_a(P)$ into 
corresponding
homogeneous components, 
 it suffices to consider the case where
$P$ is a homogeneous polynomial of degree $n$ without a term $\alpha_na^n$
and $\ad_a(P)$ is a homogeneous Lie polynomial of degree $n+1$.
From Lemma~\ref{lemma:rPa_ad_aP} it follows
\begin{equation*}
-\ad^2_a(P)=r(\ad_a(P)a)
=r(aPa)-\underbrace{r(Paa)}_{=0}=\ad_a(r(Pa))
\end{equation*}
and thus
$\ad_a(\ad_a(P)+r(Pa))=0$ such that
\begin{equation*}
\ad_a(P)=-r(Pa)
\end{equation*}
by Lemma~\ref{lemma:ad_injectivity}.\footnote{Note that for the last equation 
we cannot use Lemma~\ref{lemma:rPa_ad_aP} directly, because we do
not (yet) know that $P$ is a Lie polynomial.}
Using Lemma~\ref{lemma:rP_nP} we obtain
\begin{equation*}
(n+1)\ad_a(P)=r(\ad_a(P))=r(aP)-r(Pa)=\ad_a(r(P))+\ad_a(P)
\end{equation*}
and thus $n\,\ad_a(P)=\ad_a(r(P))$ from which it follows
\begin{equation*}
\ad_a\left(P-\frac{1}{n}r(P)\right)=0.
\end{equation*}
Using Lemma~\ref{lemma:ad_injectivity} we obtain
$P=\frac{1}{n}r(P)$ where the right-hand side is obviously a
Lie polynomial.
\end{proof}

%\begin{biog}
%\item[Harald Hofst\"{a}tter] received his Ph.D. in mathematics 
%from Vienna University of Technology in 2000. 
%\begin{affil}
%Reitschachersiedlung 4/6, A--7100 Neusiedl am See, Austria\\
%hofi@harald-hofstaetter.at
%\end{affil}
%\end{biog}
\vfill\eject

\end{document}